\documentclass[amssymb,amscd,10pt]{article}
\usepackage{epsfig}
\usepackage{amsmath}
\usepackage{amsthm}
\usepackage{amssymb}

\usepackage{amscd}
\usepackage{amsfonts}  
\usepackage{makeidx}
\usepackage{hyperref}
\usepackage[all]{xy}
\usepackage{graphicx}    
\usepackage{multicol}    

 \hypersetup{colorlinks=true,
urlcolor=blue, citecolor=blue, linkcolor=blue}

\DeclareMathAlphabet\EuR{U}{eur}{m}{n}
\SetMathAlphabet\EuR{bold}{U}{eur}{b}{n}

\newcounter{commentcounter}
\newcommand{\comment}[1]                      
{\stepcounter{commentcounter}
{\bf Comment \a rabic{commentcounter}}: {\ttfamily #1} }

\newcommand{\id}{\operatorname{id}}

\newtheorem{theorem}{Theorem}

\newtheorem{definition}[theorem]{Definition}

{\catcode`@=11\global\let\c@equation=\c@theorem}

\title {A zoo of geometric homology theories}
\author{Matthias Kreck }

\begin{document}
\maketitle

\section{Introduction} A homology theory is on the one hand given by a spectrum - and from this point of view homology theories are almost as general as spaces. Originally they occurred in a completely different form by geometric constructions like simplicial or singular homology theories or later bordism theories, $K$-theory (a cohomology theory) and others. In this note we introduce a zoo of homology theories which both generalize singular homology and bordism theory in a natural way. More precisely for each subset $A$ of the natural numbers $\mathbb N$ we construct a homology theory 
$h^A_*$ which for $A = \mathbb N - \{1\}$ is ordinary singular homology and for $A =\{0\}$ is singular bordism. 

The theories in our zoo are all bordism groups, which generalize the case of smooth manifolds by allowing singularities. There are many concepts of manifolds with singularities one could use here. For our pupose the objects the author introduced some years ago, which are called {\bf stratifolds}, work particularly well \cite{K}. The theory of stratifolds was further elaborated in \cite{Gr} in the thesis of the author's PhD student Anna Grinberg. The zoo comes from forcing certain strata indexed by the subset $A$ to be empty.

Despite their simple construction computations of these groups seem to be very complicated. We give a few simple examples. Thus there are no interesting applications so far and the zoo looks a bit like a curiosity. But one never knows for what these theories might be good in the future. We mention a concrete question which might be useful in connection with the Griffith group consisting of algebraic cycles in a smooth algebraic variety over the complex numbers which vanish in singular homology. 

I dedicate these notes to my friend Egbert Brieskorn. Egbert is (in a very different way like our common teacher Hirzebruch)  a person which had a great influence on me. When I had to make a complicated decision I often had him in front of my eyes and asked myself: What would Egbert suggest? Conversations with him were always  intense and  fruitful. I miss him very much. 

When I thought about a subject for this note I also asked myself, what would Egbert say about this or that mathematics. I have no idea what he would say about this zoo. But I hope he would at least like the occurrence of manifolds with singularities. And it would probably find his interest that if $Y$ is a compact complex singular variety in a non-singular complex  algebraic variety $X$ it admits a natural structure of a stratifold with all odd-dimensional  strata empty and so represents a homology class in the special case where $A$ consists of all odd numbers.

\section{Generalized homology theories and singular bordism} To motivate the construction let me recall the definition of  singular bordism groups. Let $X$ be a topological space. Then a cycle is a pair $f: M \to X$, where $M$ is a closed  smooth $n$-dimensional manifold and $f$ a continuous map. Two cycles $(M,f)$ and $(M',f')$ represent the same bordism class if and only if there is a compact  manifold $W$ with $\partial W = M + M'$,  and an extension $F: W \to X$ of the maps $f$ and $f'$. This is an equivalence relation and the equivalence classes form a group under disjoint union denoted by $\mathcal N _n(X)$, the $n$-th singular bordism group. If $g: X \to Y$ is a continuous map it induces a homomorphism 
$$
g_* :\mathcal N _n (X) \to \mathcal N_n (Y)
$$
given by post-composition 
and this way we obtain for each $n$ a functor from the category of topological spaces to the category of abelian groups. By construction (using the cylinder as a bordism) this is a {\bf homotopy functor}, meaning that if $g$ and $g'$ are homotopic, then $g_* = g'_*$. 

This functor is a homology theory, which normally is expressed as an extension to the category of topological pairs fulfilling the Eilenberg-Steenrod axioms. But an equivalent simple characterization is the following. As in the case of relative homology groups one has to add data to a functor $h_*$, namely a boundary operator, which in our case is the boundary operator for a Mayer-Vietoris sequence: for open subsets $U$ and $V$ a natural operator
$$
d: h_k(U \cup V) \to h_{k-1}(U \cap V)
$$
Then a {\bf homology theory} is a homotopy functor $h_*$ together with a natural boundary operator as above, such that the Mayer-Vietoris sequence 
$$
... \to h_{k+1}(U \cup V) \to h_k(U \cap V) \to h_{k} (U) \oplus h_k(V) \to h_k(U \cup V) \to ....
$$
is exact. Here the maps are given by the boundary operator, the  induced maps of the inclusions and the difference of the induced maps of the inclusions. 

Examples of homology theories are singular homology and the bordism groups $\mathcal N_*(X)$. In this case the boundary operator is given a follows. If $f: M \to U\cup V$ is a continuous map, then consider $A := f^{-1} (U) $ and $B:= f^{-1}(V)$. These are disjoint open subsets. Thus there is a smooth function $\rho: M \to \mathbb R$, which on $A$ is $0$ and on $B$ is $1$. Let $t \in (0,1)$ be a regular value of $f$. Then $d[(M,f)]$ is represented by $f|_{f^{-1}(t)}: f^{-1}(t) \to U \cap V$. The construction of singular bordism was carried out in \cite{C-F} on the category of pairs of spaces. The proof that our absolute bordism theory is a homology theory uses the same ideas. It has nothing to do with the fact that the cycles are maps on smooth manifolds. If works identically for manifolds with singularities, where it was worked out in \cite {K}. The same arguments apply to the generalized bordism theories constructed below. 

\section{Stratifolds} There are plenty of definitions of stratified spaces, starting from Whitney stratified spaces over Mather's abstract stratified spaces, which both are differential topological concepts, to purely topological concepts. All of them is common that it is a topological space together with a decomposition into manifolds, which are called strata. Since we want to generalize bordism of smooth manifolds we restrict ourselves to differential topological stratifolds. 

Our approach to stratifolds is motivated by a definition of smooth manifolds in the spirit of algebraic geometry  as topological spaces together with a sheaf of functions, which in the traditional definition corresponds to the smooth functions. Then a manifold is a Hausdorff space $M$ with countable basis together with a sheaf $\mathcal C$ of continuous functions, which is locally diffeomorphic to $\mathbb R^n$ equipped with the sheaf of all smooth functions. Here a morphism between spaces $X$ and $X'$  equipped with subsheaves of the sheaf of smooth functions is a continuous map $f$  such that if $\rho'$ is in the sheaf over $X'$, then $\rho' f$ is in the sheaf over $X$. An isomorphism or here called {\bf diffeomorphism} is a bijetive map $f$ such that $f$ vand $f^{-1}$ is a morphism.

Having this in mind it is natural to generalize this by considering topological Hausdorff spaces $\mathcal S$ with countable basis together with a sheaf $\mathcal C$ of continuous functions, such that for $f_1,...,f_k$ in $\mathcal C$ and $f$ a smooth function on $\mathbb R^k$, the composition $f (f_1,..,f_k)$ is in $\mathcal C$.  A stratifold is defined as a pair $(\mathcal S, \mathcal C)$ such that the following  properties are fulfilled. Given $\mathcal C$ one can define the tangent space $T_x \mathcal S$ at a point $x\in \mathcal S$ as the vector space of all derivations of the germ $\Gamma_x(\mathcal C)$ of smooth functions at $x$. This gives a decomposition of $\mathcal S$ as subspaces $\mathcal S^k := \{x \in \mathcal S| \, dim T_x\mathcal S= k\}$,. These subspaces are called the {\bf $k$-strata} of $\mathcal S$. The union of all strata of dimension $\le k$ is called the {\bf $k$-skeleton}  $\Sigma ^k$. 

\begin{definition} An $n$-dimensional stratifold is a pair $(\mathcal S, \mathcal C)$ as above such that \\
(1) For all $k$ the stratum $\mathcal S^k$ together with the restriction $\mathcal S_{S^k}$ of the sheaf to it is a smooth manifold, i.e. locally diffeomorphic to $\mathbb R^k$. \\ 
 (2) All skeleta are closed subsets of $\mathcal S$. \\
 (3) The strata of dimension $>n$ are empty.\\
 (4) For each $x \in \mathcal S$  and open neighborhood $U$ there is a so called bump function $\rho :  \mathcal S \to \mathbb R_{\ge 0}$  in $\mathcal C$,  such that $supp \rho \subset  U$  and $\rho(x)  >0$. \\
(5) For each $x \in S^k$ the restriction gives an isomorphism $\Gamma_x(\mathcal C)\to \Gamma_x(\mathcal C|_{S^k})$.

A continuous map $f : \mathcal S \to \mathcal S'$ is called a morphism or  {\bf smooth}, if $f\rho \in \mathcal C$ for each $\rho \in \mathcal C'$. If $f$ is a homeomorphism and $f$ and $f'$ are smooth, the $f$ is called a {\bf diffeomorphism}. 
 \end{definition}
 
 A smooth map $f$ induces, as for smooth manifolds, a linear map between the tangent spaces, the differential. It is given by pre-composition with the map $f$ mapping a derivation at $x \in \mathcal S$ to a derivation of $\mathcal S'$ at $f(x)$. This induced map is called the {\bf differential} of $f$ at $x$.
 
 Whereas the other conditions are natural, one might wonder where the last condition comes from. If one looks at Mather's abstract stratified spaces, then he gives the decomposition of the space into the strata plus additional data. Amongst them there are neighborhoods of the strata together with retracts $\pi$ to the strata. Then Mather defines smooth (also called controlled) functions $f$ as continuous functions such that for each stratum the restriction to it is smooth and  there is a smaller neighborhood such that $\pi$ restricted to the smaller stratum commutes with $f$. This implies our condition (5) and actually one can reconstruct $\pi$ from our data, if (5) is fulfilled (\cite{K},p. 18ff). 
 
All smooth manifolds are stratifolds. In this note we will only use the following comparatively simple class of stratifolds, which is similar to the construction of $CW$-complexes, which we call {\bf polarizable stratifolds}, abbreviated as {\bf $p$-stratifolds}. A $0$-dimensional $p$-stratifold is a $0$-dimensional smooth manifold. Let $(\mathcal S, \mathcal C) $ be a $(k-1)$-dimensional $p$-stratifold and $W$ be a $k$-dimensional manifold with boundary and $f: \partial W \to \mathcal S$ a proper smooth map. Then we define a $k$-dimensional $p$-stratifold $(\mathcal S' := W \cup_f \mathcal S,\mathcal C')$, where $\mathcal C'$ is constructed as follows. Choose a collar $\varphi: \partial W \times [0,1) \to U\subset W$. Then $\rho: \mathcal C'$ is in $\mathcal C'$ if and only if $f|_\mathcal S$ and $f|_ W$ is smooth and there is an open subset $U' \subset U$ such that $f$ commutes with the retract to $\partial W$ given by  the collar. The last condition guarantees condition (5) above. It is easy to check that this is a $k$-dimensional stratifold. 

This way one obtains plenty of explicit stratifolds. For example if $W$ is a compact manifold with boundary and $f$ the constant map from the boundary to a point. Then if we choose a collar of the boundary and attach $W$ to the point (equivalently collaps the boundary to a point) and define the sheaf as above, we obtain a stratifold with $0$-stratum a point and top-stratum the interior of $W$. A special case of this is the open cone over a smooth manifold. 

If $\mathcal S$ is an $n$-dimensional $p$-stratifold and $M$ is a $m$-dimensional smooth manifold then the product $\mathcal S \times M$  is naturally an $(n+m)$-dimensional $p$-stratifold. In the construction above one replaces $W$ by $W \times M$ and the attaching map by $f \times \id$.  

We define {\bf an $n$-dimensional $p$-stratifold $\mathcal T$ with boundary} as a pair of topological spaces $(T,\partial T)$ together with the structure of  an $n$-dimensional stratifold on $ T - \partial T$, the structure of an $(n-1)$-dimensional stratifold on $\partial T$ such that there is a homeomorphism $\varphi: \partial T \times [0,1)$ onto an open neigbourhood  $U \subset T$ of $\partial T$, which on $\partial T$ is the identity, such that $T - U$ is a closed subset of $T$ (implying that $\partial T$ is an end) and its restriction to $\partial T \times (0,1)$ is a diffeomorphism of stratifolds onto $U - \partial T$. Such a homeomorphism $\varphi$ is called a {\bf collar}. 

Using a collar one can glue $p$-stratifolds the same way one glues manifolds over common boundary components. 
Thus one can define bordism groups and, if one adds a continuous map to a topological space $X$ singular bordism groups.

 The following observation is central for our construction of the zoo of bordism groups. If $\mathcal T$ and $\mathcal T'$ are $p$-stratifolds whose stratum of dimension $r$ is in both cases empty, then  the same holds for the glued stratifold. Similarly if $\mathcal S$ and $\mathcal S'$ are stratifolds with empty $k$-stratum, then the same holds for the disjoint union. Let $A \subset \mathbb N$ be a set. Here $\mathbb N$ contains $0$. An {\bf $n$-dimensional $A$-stratifold} is a $p$-stratifold $\mathcal S$ such that for $a \in \mathbb N - A$ the stratum of dimension $n-a$ is empty. For example, if $A = \{0\}$, then an $A$-stratifold is a smooth manifold, all strata except the top stratum are empty. Or, if $A =\mathbb N -  \{1\}$, then $\mathcal S$ is an $A$-stratifold, if the stratum of dimension $n-1$ is empty. Or, if $A$ consists of the even numbers, then $\mathcal S$ is an $A$-stratifold if and only if the strata of odd codimension are empty.
 
\section{ The zoo and the main theorem} 

 With this it is possible to define the zoo of bordism theories. 
 
 \begin {definition} Let $X$ be a topological space and $n$ a natural number and $A \subset \mathbb N$. An {$n$-dimensional \bf singular $A$-stratifold in $X$} is a closed (compact without boundary) $n$-dimensional  $A$-stratifold $\mathcal S$ together with a continuous map $f: \mathcal S \to X$.
 
 A {\bf singular $A$-bordism between} two $n$-dimensional singular $A$-stratifolds $(\mathcal S ,f)$ and $(\mathcal S, f')$ is a compact singular $A$-stratifold $\mathcal T$ with boundary $\mathcal S + \mathcal S'$ together with an extension $F: \mathcal T \to X$ extending $f$ and $f'$. 
 
 \end{definition} 
 
 Since one can glue $n$-dimensional singular $A$-stratifolds over  common boundary components, singular $A$-bordant is an equivalence relation. Thus one can consider the equivalence classes, which form a group under disjoint union denotes by $\mathcal N ^A_n(X)$. The prof is the same as in the case of smooth manifolds. 
 
 If $g: X \to Y$ is a continuous map, the post-composition induces a homomorphism $\mathcal N ^A_n(X)\to \mathcal N ^A_n(Y)$, which makes $\mathcal N ^A_n(X)$ a {\bf functor} from the category of topological spaces and continuous maps to the category of graded abelian groups and continuous maps. 
 
 To formulate our main theorem, namely that for each $A$ we obtain a homology theory, we have to construct boundary operators. We have described above how this is done for bordism groups of smooth manifolds. To generalize this to stratifolds one has to consider {\bf regular values} of smooth maps $\rho$ from a $p$-stratifold $\mathcal S$ to $ \mathbb R$. A value $t \in \mathbb R$ is a regular value if the restriction to all strata is a regular value. We note that by definition of the sheaf $\mathcal C$, if $\mathcal S$ is constructed inductively by attaching smooth manifolds $W$ via a smooth map to the lower skeleta, $t$ is also the regular value of the restriction of $t$ to the boundary of $W$. The reason is that $\rho$ commutes with the retracts given by the collar. This implies that the preimage of $\rho$ restricted to $W$ is a smooth manifold with boundary and the restriction of the collar chosen on $W$ is a collar on this preimage. This implies that the preimage $\rho^{-1}(t)$ is in a natural way a $p$-stratifold of  codimension $1$. Furthermore, if $f: \mathcal S \to X$ is a continuous map, we can consider its restriction to $\rho^{-1}(t)$. Finally, if $\mathcal S$ is an $A$-stratifold, then the codimension of the strata of $\rho^{-1}(t)$ is unchanged and so $\rho^{-1}(t)$ is again an $A$-stratifold. 
 
 Thus one can define the boundary operator in the Mayer-Vietoris sequence as for smooth manifolds as follows.  Let $U$ and $V$ be open subsets of $X$ and  $f: \mathcal S |to X$ a singular $A$-stratifold. Then we consider the complement $A$  of $f^{-1} (U)$ and $B$ of $f^{-1}(V)$. These are closed subsets. In a stratifold one has partition of unity \cite{K}, Proposition 2.3, and so there is a smooth function $\rho$, which on $A$ is zero and on $B$ is one. In a stratifold one can apply Sard's Theorem (\cite{K}, Proposition 2.6), and so there is a regular value $t \in (0,1)$. By the considerations above $f^{-1}(t)$ is a codimension $1$ stratifold and the restriction of $f$ to it gives a singular $A$-stratifold in $U \cap V$. We will show that this is well defined and gives a natural boundary operator.

Our main Theorem is the following:

\begin{theorem} Let $A$ be a subset of $\mathbb N$. For open subsets $U$ and $V$ in a topological space $X$ the boundary operator
$$
d: \mathcal N_n^A(U\cup V) \to \mathcal N^A_{n-1}(U \cap V)
$$
is well defined and natural.

The functor $\mathcal N ^A_n(X)$ together with $d$ is a homology theory. 
\end{theorem}

\begin{proof} We first note that since a homotopy is a special bordism, the functor is a homotopy functor. Thus one only has to prove that there is an exact  Mayer-Vietoris sequence. This amounts to showing that for all open subsets $U$ and $V$ of $X$ the differential $
d: \mathcal N_n^A(U\cup V) \to \mathcal N^A_k(U \cap V)
$ is well defined and natural and that the sequence is exact. 

We begin with the proof that $d$ is well defined. In the case of bordism of smooth manifolds this is easy using that $\rho^{-1}(t)$ has a bicollar. In the case of stratifolds this is not the case. But it was shown in \cite{K}, Lemma B.1, page 197 that up to bordism one has a bicollar. This was proved there for so called regular stratifolds. The regularity was used only at one place, namely to guarantee, that the set of regular values are an open subset, if $\mathcal S$ is compact \cite{K}, Proposition 4.3, page 44. Once this is the case, then the proof of \cite{K}, Lemma B.1 goes through without any change for $p$-stratifolds. 

Thus we show that the regular values of $\rho$ are an open set, if $\mathcal S$ is a compact $p$-stratifold. For this we consider the regular points, the points in $\mathcal S$, where the differential of $\rho$ is non-trivial. But $x \in \mathcal S$ is a regular point, if and only it's restriction to the inner of attached manifold $W$ is regular. This restriction to $W$ extends to $\partial W$ and commutes with the retract given by a collar. This implies that the regular points form an open subset. If  $\mathcal S$ is compact, the image of an open subset is open and so the regular values form an open set. Thus the proof of \cite{K}, Lemma B.1 goes through for $p$-stratifolds.

With this the proof that the boundary operator is well defined is the same as in \cite{K} for regular stratifolds. The naturality follows more or less from the construction of the differential. Let $g: X \to X'$ is a continuous map and $U$, $V$ are open subsets of $X$ and $U'$ and $V'$ are open subsets of $X'$ such that $g(U) \subset U'$ and $g(V) \subset V'$. Then for a singular $A$-stratifold $f: \mathcal S \to X$ we denoted the complements of the preimage of $U$ and $V$ by $A$ and $B$. We have chosen a smooth function $\rho$, which on $A$ is $0$ and on $B$ os $1$. Now we consider $gf$ and notice that $A' \subset A$ and $B' \subset B$. This we can take the same separating function for the definition of the boundary operator $d'$.

Lemma B.1 in \cite{K} is also the key for the proof of the special case considered in \cite{K}, that the Mayer-Vietoris sequence is exact. The case considered there is the case, where $A = \mathbb N - \{1\}$. That $A$ is of that special form is nowhere used in this proof. The only thing that matters is, that all constructions used in the proof stay in the world of $A$-stratifolds. These constructions are: gluing of stratifold via parts of boundary components and taking  the preimage of a regular value. The definition of $A$ manifolds using conditions on the existence of non-empty  strata of a certain codimension are compatible with these constructions. Thus the proof in \cite{K} goes through.

\end{proof}

One can enlarge this zoo even more by adding more structure to the strata of a stratifold, for example an orientation or a stable almost complex structure or a spin-structure or a framing. In all these cases one obtains again a homology theory. 

Now we mention a few special cases which show that the $A$-homology theories give a unified picture of some of the most important homology theories which originally have rather different constructions. To formulate the result let me remind the reader of the Postnikov tower of a homology theory. As mentioned above one has a unified homotopy theoretic picture of homology theories in terms of spectra $S$. Given a spectrum $S$ and a topological space $X$ one can consider the stable homotopy groups $\pi_n(S \wedge X)$, which form a homology theory. Like with spaces one can consider Postnikov towers of spectra. This is given by a spectrum, $S_k$ together with a map $S \to S_k$, where one requires that all stable homotopy groups of $S$ vanish above degree $k$ and the map induces an isomorphism up to degree $k$. 

 If we consider for example the Thom spectrum $MO$ which represents singular bordism, then the $0$-th stage of the Postnikov tower is a homology theory, which has coefficients $\mathbb Z/2$ in degree $0$ and $0$ in degree $>0$. Thus this homology theory represents $H_k(X ; \mathbb Z/2)$. 

Returning to our zoo, we consider  some special cases. For an integer $k$ we consider the set $A_k := \mathbb N - \{1,...,k+1\}$. For $k = \infty$ we define $A_k = \{0\}$.  Then for $n\le k$ an $n$-dimensional $A_k$-stratifold is the same as a smooth manifold and so for $n\le k$ the bordism group $\mathcal N^{A_k}_n $ is equal to the bordism group of manifolds $\mathcal N_k$, whereas for $n\ge k$ the group is zero, since the cone over such a stratifold is a zero bordism. Thus we see:

\begin{theorem} The homology theory $\mathcal N^{A_k}_* $ is equivalent to the homology group given by the $k$-stage of the Postnikov tower of the Thom spectrum $MO$. In particular
$$
\mathcal N^{A_1}_*(X)
$$ 
is equivalent to 
$$
H_*(X;\mathbb Z/2),
$$
and 
$$
\mathcal N^{A_\infty}_*(X) = \mathcal N_*(X).
$$
\end{theorem}

We finish this note with a  potential application of our theories to the Griffith group. As mentioned before, one can add more structure to the strata of an $A$-stratifold. If we distinguish a stable almost complex structure on all strata we call the corresponding homology theory
$$
\mathcal U
 ^A_*(X).
 $$
In the discussion above one obtains similar statements if one replaces non-oriented bordism by unitary bordism $U_*(X)$, $MO$ by $MU$ and $H_k(X;\mathbb Z/2)$ by $H_k(X;\mathbb Z)$. 

Now,  we consider the special case of an $A$-homology theory for $A_{even}$-stratifolds with stable almost complex strata, where $A_{even}$ consists of the even natural numbers. We have a forgetful transformation (replace $A_{even}$ by $\mathcal N - \{1\}$ and use the orientation given by the almost complex structure to obtain an element in integral homology)
$$
\varphi: \mathcal U^{A_{even}}_{2r} \to H_{2r}(X;\mathbb Z).
$$

{\bf Question:} {\em What is the image and kernel of $\varphi$?}\\

This might be useful in connection with the Griffith group consisting of the kernel of the natural transformation $H: Z_{alg}^*X \to  H^*(X,\mathbb Z)$ (the letter $H$ stands for Hodge), where  $X$ is a non-singular complex algebraic variety and $Z_{alg}^*X$ is the ring of cycles modulo algebraic equivalence on $X$. For simplicity we assume that $X$ is compact, so that Poincar\'e duality holds and we can consider the corresponding map in homology  $Z^{alg}_*X \to  H_*(X,\mathbb Z)$. Totaro \cite{T} has constructed a canonical lift of this transformation over $U_*(X) \otimes _{U_*} \mathbb Z$. We will construct another lift. 

Since a complex  algebraic variety is in a natural way an $A_{even}$ $p$-stratifold \cite{Gr},  we obtain a transformation
$$
Z^{alg}_*X \to  \mathcal U^{A_{even}}_*(X).
$$
If we compose this with the transformation given by the forgetful map $\varphi$ above,
the composition of these two transformations is the Pincar\'e dual of the transformation $H:Z_{alg}^*X \to  H^*(X,\mathbb Z)$. Thus one might try to do the same as Totaro did, to find elements in the kernel of $
\mathcal U^{A_{even}}_*(X) \to H_{*}(X;\mathbb Z)
$ which are in the image of $
Z^{alg}_*X \to  \mathcal U^{A_{even}}_*(X)
$. Unfortunately we have nothing to say about this at the moment. The reason why this might be interesting is that in contrast to $U_*(X) \otimes _{U_*} \mathbb Z$ our theory $U^{A_{even}}_*(X)$ is a homology theory, which might be a useful fact. On the other hand a computation of $U^{A_{even}}_*(X) $ is probably very hard.

 \end{document}